\documentclass[11pt]{dgpaper}

\begin{document}

\title{\sc Removing a ray from a noncompact symplectic manifold}
\author{Xiudi Tang}
\date{\today}
\maketitle

\begin{abstract}
  We prove that any noncompact symplectic manifold which admits a properly embedded ray with a wide neighborhood is symplectomorphic to the complement of the ray by constructing an explicit symplectomorphism in the case of the standard Euclidean space.
  We use this excision trick to construct a nowhere vanishing Liouville vector fields on every cotangent bundle.
\end{abstract}

\section{Introduction} \label{sec:intro}

Exotic symplectic structures are known to exist on $\R^{2n}$ by Gromov \cite{MR809718}.
They can not be embedded into $(\R^{2n}, \omega_0)$.
Here $\R^{2n} = \Set{(x_1, y_1, \dotsc, x_n, y_n)}$ is the $2n$-dimensional Euclidean space and $\omega_0 = \der x_1 \wedge \der y_1 + \dotsb + \der x_n \wedge \der y_n$ is the standard symplectic form.
Can non-standard symplectic structures arise from subspaces?
An open subset $A \subset \R^{2n}$ is called \emph{standard} if $(A, \Res{\omega_0}_A)$ is symplectomorphic to $(\R^{2n}, \omega_0)$.
Finite-volume subsets are not standard.
Gromov's nonsqueezing theorem \cite{MR809718} showed that the cylinder $\Set{(x_1, y_1, \dotsc, x_n, y_n) \mmid x_1^2 + y_1^2 < 1}$ is not standard since it contains no ball of radius larger than one. 
Later, such quantitative theory called \emph{symplectic capacities} has been developed vastly, see for instance \cite{MR2369441,MR978597}.
Since capacities are symplectic invariants and $(\R^{2n}, \omega_0)$ has infinite capacity, any subset with a finite capacity will not be standard.
Nevertheless, we want to take a look at the ``large'' subsets.
For any $A \subset \R^{2n}$ diffeomorphic to $\R^{2n}$, if there are symplectic embeddings 
\begin{equation*}
  (\R^{2n}, \omega_0) \hookrightarrow (A, \Res{\omega_0}_A) \hookrightarrow (\R^{2n}, \omega_0),
\end{equation*}
is $A$ neccessarily standard?
McDuff \cite{MR906899} gave some positive examples and a possibly negative one for the above question.
Later McDuff--Traynor \cite{MR1297135} and Traynor \cite{MR1253616} showed that there are infinitely many such $A$ that can be distinguished by the spaces of isotopy classes of symplectic embeddings.
In this paper, we provide another positive example.
Let $L_1$ be the image of $\gamma_0 \colon [0, \infty) \to \R^{2n}$,
\begin{equation*}
  \gamma_0(s) = (0, \dotsc, 0, s, 0),
\end{equation*}
and let $M_1 = \R^{2n} \setminus L_1$.
We use $\mathbb{B}_x^d(r)$ to denote the open ball centered at $x$ in $\R^d$ with radius $r$.
We prove
\begin{theorem} \label{thm:main}
  The subset $M_1$ is standard.
  Moreover, for any $\varepsilon > 0$ there is a symplectomorphism $\varphi \colon (M_1, \Res{\omega_0}_{M_1}) \to (\R^{2n}, \omega_0)$ which is the identity outside of $\mathbb{B}_0^{2n-2}(\varepsilon) \times W(\varepsilon)$, where $W(\varepsilon) = \mathbb{B}_0^2(\varepsilon) \cup \Set{(x, y) \in \R^2 \mmid x > \frac{\sqrt{2}\varepsilon}{2}, x\abs{y} < \frac{\varepsilon^2}{2}}$.
\end{theorem}

Let $(M, \omega)$ be a noncompact symplectic manifold.
For a proper embedding $\gamma$ of the \emph{ray} $[0, \infty)$ into $M$, we define the \emph{$\gamma$-width} of $(M, \omega)$ as
\begin{equation*}
  \sup \Set{\pi\varepsilon^2 > 0 \mmid \exists\, j \text{ such that } \Pa{\mathbb{B}_0^{2n-2}(\varepsilon) \times \R^2, \omega_0} \xhookrightarrow{j} (M, \omega), j \circ \gamma_0 = \gamma}.
\end{equation*}
If $(M, \omega)$ has positive $\gamma$-width we call it \emph{$\gamma$-wide}.
By \cref{prop:embedding-R2}, it is equivalent to the existence of an embedding $\Pa{\mathbb{B}_0^{2n-2}(\varepsilon) \times W(\varepsilon), \omega_0} \xhookrightarrow{j} (M, \omega), j \circ \gamma_0 = \gamma$ for some $\varepsilon > 0$.
We have a corollary of \cref{thm:main} which partially answers \cite[Question~11.2]{2018arXiv181111109B}.

\begin{corollary} \label{cor:noncompact-removing-a-ray}
  Let $(M, \omega)$ be a noncompact symplectic manifold which is $\gamma$-wide for a proper embedding $\gamma$ of $[0, \infty)$. 
  Then $(M, \omega)$ is symplectomorphic to $(M \setminus \image \gamma, \Res{\omega}_{M \setminus \image \gamma})$.
\end{corollary}

Suppose we have an object on a noncompact symplectic manifold $(M, \omega)$ with only a few ``bad'' points on a properly embedded ray $\gamma$, and $(M, \omega)$ is $\gamma$-wide.
We can swipe these ``bad'' points off $M$ using the symplectomorphism constructed in \cref{cor:noncompact-removing-a-ray} to get another object which is ``regular''.
For instance, let $X$ be a Liouville vector field on an exact symplectic manifold $(M, \omega)$ and suppose we can connect all zeroes of $X$ through a properly embedded ray $\gamma$ such that $(M, \omega)$ is $\gamma$-wide.
Then $(M, \omega)$ admits a nonvanishing Liouville vector field, or equivalently a nonvanishing primitive of $\omega$, see \cite[Question~11.1]{2018arXiv181111109B}.
By removing multiple rays, we show that it is the case when $M$ is a cotangent bundle.

\begin{corollary} \label{cor:nonvanishing-Liouville-cotangent-bundle}
  The cotangent bundle of a smooth manifold admits a nonvanishing Liouville vector field.
\end{corollary}

\emph{Acknowledgments:} The author would like to thank Alan Weinstein who brought this problem up to the author in July 2017,  Dusa McDuff, and Laura Starkston for their interest.
The author thanks \'Alvaro Pelayo, Reyer Sjamaar, Alan Weinstein, and Xiangdong Yang for helpful discussions and comments.
The author thanks Sean Curry for bringing out \cref{prop:embedding-R2} in their discussions.
The majority of the work has been done at Cornell University.

\section{Equivalence via integrable systems} \label{sec:equivalence}

The proof of \cref{thm:main} relies on the equivalence of integrable systems.
A basic fact in symplectic geometry is that diffeomorphic manifolds have symplectomorphic cotangent bundles.
Note that the fibers of a cotangent bundle are Lagrangian.
In many cases, a diffeomorphism between the bases of two singular Lagrangian fibrations preserving certain structures can be uniquely lifted to a fiberwise symplectomorphism, up to the Hamiltonian flow of momentum maps (the fiberwise translation).
The principle, reducing symplectomorphisms between $2n$-dimensional manifolds to diffeomorphisms between $n$-dimensional spaces, has been used in \cite{MR984900,MR596430,MR3371718,MR2670163,2018arXiv180300998P,MR2534101,MR2024634,MR2018823}.

In \cref{ssec:integrable-systems}, we construct integrable systems, one on either $M_1$ or $\R^{2n}$, and two on a new region $M_2 \subset \R^{2n}$.
The integrable systems are \emph{semitoric}, in the sense that the first $(n-1)$ components of the momentum maps generate $\mathbb{T}^{n-1}$-Hamiltonian actions.
However, they differ from those in \cite{MR2534101,MR2784664} as the fibers are noncompact.
We prove in \cref{ssec:equivalence} an equivalence lemma for those noncompact semitoric systems which as far as we know is not in the literature.
The completeness of the Hamiltonian vector fields is crucial.
In \cref{ssec:construct-symplecto} we use the equivalence lemma to relate two systems on $M_2$ to that on $M_1$ and that on $\R^{2n}$, respectively.
In this way we show that $M_2$ is symplectomorphic to $M_1$ and $\R^{2n}$. 
As a result, $M_1$ is standard.

\subsection{Four integrable systems on three manifolds} \label{ssec:integrable-systems}

In this subsection, we establish two lemmas to prove \cref{thm:main}.
By \cref{lem:first-symplecto} we show that $M_1$ is symplectomorphic to $M_2$ (defined later) and by \cref{lem:second-symplecto} $M_2$ is symplectomorphic to $\R^{2n}$.
The supports of the symplectomorphisms are also taken care of so that $M_1$ is symplectomorphic to $\R^{2n}$ with the complement of the prescibed neighborhood of $L_0$ fixed.

Consider the momentum map of the harmonic oscillator
\begin{align*}
  \mu_0 \colon \R^2 &\to \R, \\
  (x, y) &\mapsto \frac{x^2 + y^2}{2}.
\end{align*}
See \cref{fig:mu0}.

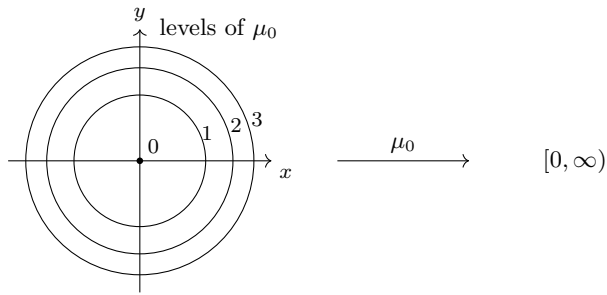
\inputfigure{mu0}{mu0}{Illustration of $\mu_0$ by its level sets.}

First, let $\mu_1 \colon \R^2 \to \R$ be a function satisfying \cref{ass:mu-1} below.
Recall that for any smooth function $f \colon M \to \R$ on a symplectic manifold $(M, \omega)$, its \emph{Hamiltonian vector field} is defined as $X_f = -\omega^{-1} (\der f) \in \mathfrak{X}(M)$.
Let $L = [0, \infty) \times \Set{0}$.
Write $\mathbb{B}^d = \mathbb{B}_0^d(1)$ and $W = W(1)$.

\begin{assumption} \label{ass:mu-1}
\ 
\begin{enumerate} [label={(\arabic*)}]
  \item The continuous function $\mu_1$ is smooth on $\R^2 \setminus L$;
  \item the range of $\mu_1$ is $[0, \infty)$;
  \item the level set $\mu_1^{-1}(0) = L$, and $\mu_1^{-1}([0, 2)) \subset W$;
  \item the Hamitonian vector field $X_{\mu_1}$ is nonvanishing and complete on $\R^2 \setminus L$.
\end{enumerate}
\end{assumption}

Let
\begin{align*}
  F_1 = (\mu_0, \dotsc, \mu_0, \mu_1) \colon \R^{2n} &\to \R^n,
\end{align*}
and let
\begin{align*}
  B_1 &= [0, \infty)^n \setminus \Set{0}^n, \\
  B_2 &= [0, \infty)^n \setminus \Set{0}^{n-1} \times [0, 1].
\end{align*}
Let
\begin{equation*}
  L_2 = \Set{(0, \dotsc, 0, x_n, y_n) \in \R^{2n} \mmid \mu_1(x_n, y_n) \leq 1} \subset \mathbb{B}^{2n-2} \times W
\end{equation*}
and let $M_2 = \R^{2n} \setminus L_2$.
Then $M_1 = F_1^{-1}(B_1)$ and $M_2 = F_1^{-1}(B_2)$.
Moreover, $(M_1, \Res{\omega_0}_{M_1}, \Res{F_1}_{M_1})$ is an intergrable system over $B_1$ and $(M_2, \Res{\omega_0}_{M_2}, \Res{F_1}_{M_2})$ is an intergrable system over $B_2$ (except where $F_1$ is not smooth).
Observe that $B_1$ and $B_2$ are diffeomorphic.
We prove the following lemma in \cref{ssec:construct-symplecto}.

\begin{lemma} \label{lem:first-symplecto}
  Two systems $(M_1, \Res{\omega_0}_{M_1}, \Res{F_1}_{M_1})$ and $(M_2, \Res{\omega_0}_{M_2}, \Res{F_1}_{M_2})$ are fiberwise symplectomorphic.
  In particular, $(M_1, \Res{\omega_0}_{M_1})$ and $(M_2, \Res{\omega_0}_{M_2})$ are symplectomorphic with the complement of $\mathbb{B}^{2n-2} \times W$ fixed.
\end{lemma}

Second, let $\mu_2 \colon \R^2 \to \R$ be a smooth function satisfying \cref{ass:mu-2}.

\begin{assumption} \label{ass:mu-2}
\ 
\begin{enumerate} [label={(\arabic*)}]
  \item On $\mu_1^{-1}([1, \infty))$, $\mu_2 = \mu_1 - 1$;
  \item the range of $\mu_2$ on $\mu_1^{-1}([0, 1))$ is $(-\infty, 0)$;
  \item the Hamitonian vector field $X_{\mu_2}$ is nonvanishing and complete.
\end{enumerate}
\end{assumption}

Let
\begin{align*}
  F_2 = (\mu_0, \dotsc, \mu_0, \mu_2) \colon \R^{2n} &\to \R^n.
\end{align*}
and let
\begin{align*}
  B_3 &= [0, \infty)^{n-1} \times \R \setminus \Set{0}^{n-1} \times (-\infty, 0], \\
  B_4 &= [0, \infty)^{n-1} \times \R.
\end{align*}
Then $M_2 = F_2^{-1}(B_3)$ and $F_2$ is onto $B_4$.
Moreover, $(M_2, \Res{\omega_0}_{M_2}, \Res{F_2}_{M_1})$ is an intergrable system over $B_3$ and $(\R^{2n}, \omega_0, F_2)$ is an intergrable system over $B_4$.
Observe that $B_3$ and $B_4$ are diffeomorphic.
We prove the following lemma in \cref{ssec:construct-symplecto}.

\begin{lemma} \label{lem:second-symplecto}
  Two systems $(M_2, \Res{\omega_0}_{M_2}, \Res{F_2}_{M_2})$ and $(\R^{2n}, \omega_0, F_2)$ are fiberwise symplectomorphic.
  In particular, $M_2 \subset \R^{2n}$ is standard via a symplectomorphism fixing the complement of $\mathbb{B}^{2n-2} \times W$.
\end{lemma}

\subsection{Equivalence lemma} \label{ssec:equivalence}

We want to consider \emph{noncompact semitoric integrable systems} satisfying \cref{ass:noncompact-semitoric-integrable-system}.

\begin{assumption} \label{ass:noncompact-semitoric-integrable-system}
The triple $(M, \omega, F)$ is constructed as follows:
\begin{enumerate} [label={(\arabic*)}]
  \item let $\mu \colon \R^2 \to \R$ be a continuous map and $I \subset \mu(\R^2)$ is an interval such that on $\mu^{-1}(I)$, $\mu$ is smooth and $X_\mu$ is nonvanishing and complete;
  \item let $\tilde{F} = (\mu_0, \dotsc, \mu_0, \mu) \colon \R^{2n} \to \R^n$;
  \item let $B$ be a subset of $\tilde{F}(\R^{2n-2} \times \mu^{-1}(I)) = [0, \infty)^{n-1} \times I$, called the \emph{base}, such that $M = \tilde{F}^{-1}(B) \subset \R^{2n}$ is open;
  \item then let $\omega = \Res{\omega_0}_M$ and $F = \Res{\tilde{F}}_M$.
\end{enumerate}
\end{assumption}

Note that the map $\Res{\mu}_{\mu^{-1}(I)}$ admits smooth sections, for instance, the ones given by the flow of $\nabla \mu / \abs{\nabla \mu}^2$, where $\nabla \mu$ is the gradient of $\mu$ with respect to the Euclidean metric.
We will establish the following equivalence lemma.

\begin{lemma} \label{lem:equivalence}
  Let $(M, \omega, F)$ and $(M', \omega', F')$ be two noncompact semitoric integrable systems satisfying \cref{ass:noncompact-semitoric-integrable-system} with the same $\mu$ and $I$.
  Let $B = F(M)$ and $B' = F'(M')$ be their bases.
  Let $s_n \colon I \to \R^2$ be a smooth section of $\Res{\mu}_{\mu^{-1}(I)}$.
  Let $s \colon B \to M$ be $s(c_1, \dotsc, c_{n-1}, c_n) = (c_1, \dotsc, c_{n-1}, s_n(c_n))$.
  If $G \colon B \to B'$ is a diffeomorphism such that $G(c_1, \dotsc, c_{n-1}, c_n) = (c_1, \dotsc, c_{n-1}, g(c_1, \dotsc, c_n))$ for some $g \colon B \to \R$, then there is a unique symplectomorphism $\varphi \colon M \to M'$ preserving $s$ and lifting $G$, namely $\varphi \circ s = s$ and $F' \circ \varphi = G \circ F$.
  Moreover, if $G$ is the identity in a subset $C \subset B$ without isolated points, then $\varphi$ is the identity in $F^{-1}(C)$.
\end{lemma}

\begin{proof}
  For $x = (z_1, \dotsc, z_n) \in M$, let $b = (c_1, \dotsc, c_n) = F(x) \in B$.
  Here we use the identification $z_i = x_i + y_i \sqrt{-1}$ for $i = 1, \dotsc, n$.
  Let $\Phi_{\mu}^t \colon \mu^{-1}(I) \to \mu^{-1}(I)$, $t \in \R$ be the flow of $X_\mu$.
  Since $X_\mu$ is nonvanishing and complete, 
  \begin{align*}
    \R \times I &\to \mu^{-1}(I), \\
    (t, c_n) &\mapsto \Phi_{\mu}^t(s_n(c_n))
  \end{align*}
  is a diffeomorphism.
  Let $t_n \colon \mu^{-1}(I) \to \R$ be the smooth submersion defined by $z_n = \Phi_{\mu}^{t_n(z_n)}(s_n(c_n))$ for any $z_n \in \mu^{-1}(I)$.
  Let $h \colon B' \to \R$ be the $n$-th component of $G^{-1}$.
  Define
  \begin{align*}
    \varphi \colon M &\to M', \\
    \Pa{z_1, \dotsc, z_{n-1}, z_n} &\mapsto \Pa{e^{\frac{\partial h}{\partial c_1} t_n \sqrt{-1}} z_1, \dotsc, e^{\frac{\partial h}{\partial c_{n-1}} t_n \sqrt{-1}} z_{n-1}, \Phi_{\mu}^{\frac{\partial h}{\partial c_n} t_n} s_n(g(b))}.
  \end{align*}
  Here, $t_n$ is short for $t_n(z_n)$ and $\frac{\partial h}{\partial c_i}$ for $i = 1, \dotsc, n$ is short for
  \begin{equation*}
    \frac{\partial h}{\partial c_i}(G(b)) = \begin{cases}
      -\frac{\partial g/\partial c_i(b)}{\partial g/\partial c_n(b)}, & i = 1, \dotsc, n-1; \\
      \frac{1}{\partial g/\partial c_n(b)}, & i = n.
    \end{cases}
  \end{equation*}
  Then $\varphi$ lifts $G$.
  It is bijective since, by the nonvanishing and completeness of $X_\mu$ on $M$ and $M'$, it is bijective $F^{-1}(b) \to (F')^{-1}(G(b))$ for each $b \in B$.
  It is smooth by the explicit formula. 
  We have $\varphi^* \omega' = \omega$ by direct calculations or noting that it preserves a Lagrangian section $s$ and is equivariant under Hamiltonian flows, see \cite[Lemma~2.3]{2018arXiv180300998P}.
  The uniqueness is also by the equivariance under Hamiltonian flows.

  Moreover, if $\Res{G}_C = \identity_C$, by the definition of $\varphi$ in $F^{-1}(C^\circ)$, it is the identity there, so $\varphi$ is the identity in $F^{-1}(C)$ by continuity.
\end{proof}

\subsection{Constructing symplectomorphisms} \label{ssec:construct-symplecto}

Let $\ce = c_1 + \dotsb + c_{n-1}$. 
We prove, in \cref{ssec:diffeomorphism}, that for the open neighborhood $U_1 = \Set{(c_1, \dotsc, c_n) \in B_1 \mmid c_n < 2, \ce < 1, \ce < c_n}$ of $\Set{0} \times (0, 1]$ in $B_1$, there is a smooth function $g_1 \colon B_1 \to \R$, which satisfies the following assumption.
\begin{assumption} \label{ass:g1}
\ 
\begin{enumerate} [label={(\arabic*)}]
  \item $g_1(c_1, \dotsc, c_n) = c_n$ in $B_1 \setminus U_1$;
  \item $\frac{\partial g_1}{\partial c_n}(c_1, \dotsc, c_n) > 0$;
  \item $\displaystyle\lim_{c_n \to 0^+} g_1(0, \dotsc, 0, c_n) = 1$, and $g_1(c_1, \dotsc, c_{n-1}, 0) = 0$ for $\ce > 0$;
  \item $\displaystyle\lim_{c_n \to \infty} g_1(c_1, \dotsc, c_n) = \infty$.
\end{enumerate}
\end{assumption}

Then $G_1 \colon B_1 \to B_2$ defined as $G_1(c_1, \dotsc, c_{n-1}, c_n) = (c_1, \dotsc, c_{n-1}, g_1(c_1, \dotsc, c_n))$ is a diffeomorphism.

\begin{proof}[Proof of \cref{lem:first-symplecto}]
  Let $V_1 = \Set{(c_1, \dotsc, c_n) \in B_1 \mmid \ce < 2c_n} \supset \overline{U_1}$.
  The closure is taken in $B_1$.
  Note that $F_1$ is smooth on $F_1^{-1}(V_1)$.
  By applying \cref{lem:equivalence} to 
  \begin{equation*}
    \Pa{F_1^{-1}(V_1), \Res{\omega_0}_{F_1^{-1}(V_1)}, \Res{F_1}_{F_1^{-1}(V_1)}}~\text{and}~\Pa{F_1^{-1}(G_1(V_1)), \Res{\omega_0}_{F_1^{-1}(G_1(V_1))}, \Res{F_1}_{F_1^{-1}(G_1(V_1))}}
  \end{equation*}
  (subsets of $M_1$ and $M_2$), a smooth section $s_{n, 1} \colon (0, \infty) \to \R^2$ of $\Res{\mu_1}_{\R^2 \setminus L}$, and the diffeomorphism $\Res{G_1}_{V_1} \colon V_1 \to G_1(V_1)$, we obtain a fiberwise symplectomorphism $\varphi_{\frac12} \colon F_1^{-1}(V_1) \to F_1^{-1}(G_1(V_1))$.
  Since $G_1 = \identity$ on $B_1 \setminus U_1$, $\varphi_{\frac12}$ is the identity in $F_1^{-1}(V_1 \setminus U_1)$.
  So, we can extend $\varphi_{\frac12}$ by the identity to a fiberwise symplectomorphism $\varphi_1 \colon M_1 \to M_2$, which coincides with the identity outside of $F_1^{-1}(U_1) \subset \mathbb{B}^{2n-2} \times W$.
\end{proof}

For an illustration of $\varphi_1$ and $G_1$, see \cref{fig:varphi1}.

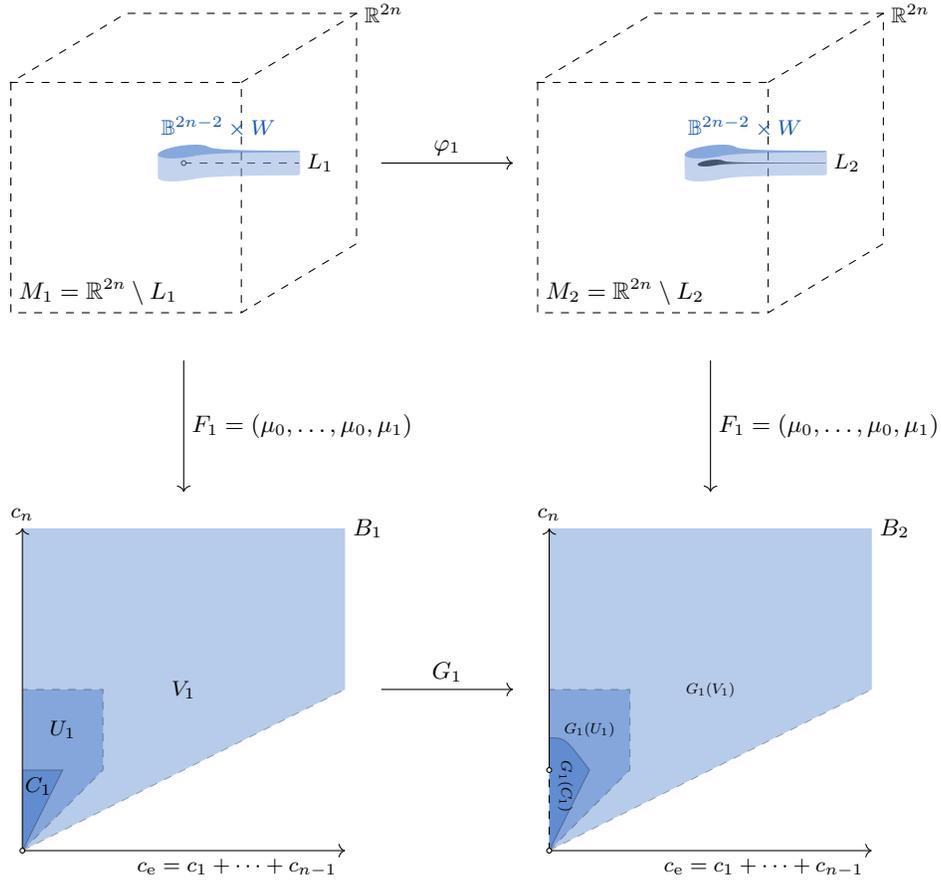
\inputfigure{varphi1}{varphi1}{The fiber preserving symplectomorphism $\varphi_1 \colon M_1 \to M_2$ supported in $\overline{\mathbb{B}^{2n-2} \times W}$ lifts the diffeomorphism $G_1 \colon B_1 \to B_2$ supported in $\overline{U_1}$.
The closed set $C_1$ is used to construct $G_1$, see \cref{ssec:diffeomorphism}.}

\bigskip

Let $U_2 = \Set{(c_1, \dotsc, c_n) \in B_3 \mmid \ce < 1, c_n < 1}$.
Then $U_2$ is an open neighborhood of $\Set{0} \times (-\infty, 0]$ in $B_3$.
We prove, in \cref{ssec:diffeomorphism}, that there is a smooth function $g_3 \colon B_3 \to \R$ satisfying the following assumption.

\begin{assumption} \label{ass:g3}
\ 
\begin{enumerate} [label={(\arabic*)}]
  \item $g_3(c_1, \dotsc, c_n) = c_n$ in $B_3 \setminus U_2$;
  \item $\frac{\partial g_3}{\partial c_n}(c_1, \dotsc, c_n) > 0$;
  \item $\displaystyle\lim_{c_n \to 0^+} g_3(0, \dotsc, 0, c_n) = \lim_{c_n \to -\infty} g_3(c_1, \dotsc, c_n) = -\infty$ if $\ce > 0$;
  \item $\displaystyle\lim_{c_n \to \infty} g_3(c_1, \dotsc, c_n) = \infty$.
\end{enumerate}
\end{assumption}

Then $G_3 \colon B_3 \to B_4$ defined as $G_3(c_1, \dotsc, c_{n-1}, c_n) = (c_1, \dotsc, c_{n-1}, g_3(c_1, \dotsc, c_n))$ is a diffeomorphism.

\begin{proof}[Proof of \cref{lem:second-symplecto}]
  By applying \cref{lem:equivalence} to 
  \begin{equation*}
    \Pa{M_2, \Res{\omega_0}_{M_2}, \Res{F_2}_{M_1}}~\text{and}~\Pa{\R^{2n}, \omega_0, F_2},
  \end{equation*}
  a smooth section $s_{n, 2} \colon \R \to \R^2$ of $\mu_2$ and the diffeomorphism $G_3 \colon B_3 \to B_4$, we obtain a fiberwise symplectomorphism $\varphi_3 \colon M_2 \to \R^{2n}$ which coinsides with the identity outside of $F_2^{-1}(U_2) \subset \mathbb{B}^{2n-2} \times W$.
\end{proof}

For an illustration of $\varphi_3$ and $G_3$, see \cref{fig:varphi3}.

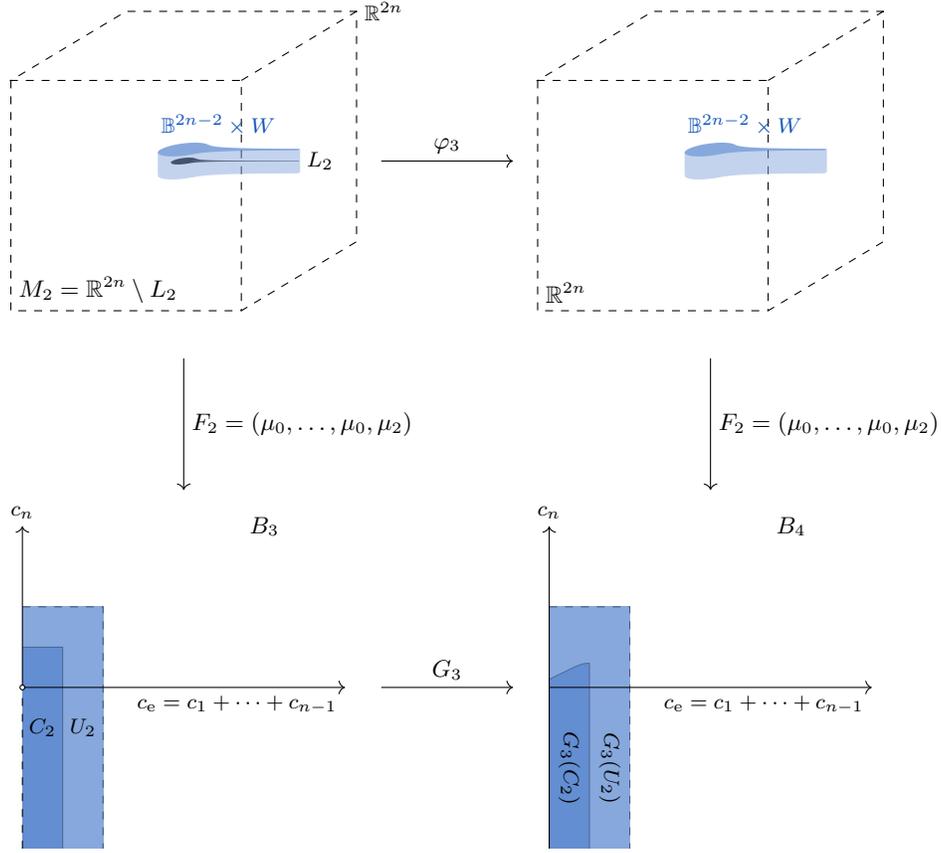
\inputfigure{varphi3}{varphi3}{The fiber preserving symplectomorphism $\varphi_3 \colon M_2 \to \R^{2n}$ supported in $\overline{\mathbb{B}^{2n-2} \times W}$ lifts the diffeomorphism $G_3 \colon B_3 \to B_4$ supported in $\overline{U_2}$.
The closed set $C_2$ is used to construct $G_1$, see \cref{ssec:diffeomorphism}.}

\section{Explicit constructions} \label{sec:explicit}

Throughout \cref{sec:equivalence} we utilize some momentum maps and diffeomorphisms satisfying certain assumptions.
To ensure those to exist, we give explicit constructions in this section.

\subsection{Momentum maps} \label{ssec:momentum-map}

In order to find $\mu_1$ satisfying \cref{ass:mu-1}, we ``blow up'' $\R^2$ along $L = [0, \infty) \times \Set{0}$ via the ``normalized double-argument map''
\begin{align*}
  \psi \colon \mathbb{H}^2 = \R \times [0, \infty) &\to \R^2, \\
  (\xi, \eta) &\mapsto \Pa{\frac{\xi^2-\eta^2}{\sqrt{2(\xi^2+\eta^2)}}, \frac{2\xi\eta}{\sqrt{2(\xi^2+\eta^2)}}}.
\end{align*}
Then $\psi$ restricts to a symplectomorphism $\R \times (0, \infty) \to \R^2 \setminus L$ with respect to the standard symplectic form and folds $\Set{0} \times \R$ along the origin onto $L$, 
In particular, the Hamiltonian dynamics on $\R \times (0, \infty)$ and $\R^2 \setminus L$ are equivalent.

Recall that $W = \mathbb{B}^2 \cup \Set{(x, y) \in \R^2 \mmid x > \sqrt{\frac12}, x\abs{y} < \frac12}$ is an open neighborhood of $L$ in $\R^2$, so $\psi^{-1}(W)$ is an open neighborhood of $\Set{0} \times \R$ in $\mathbb{H}^2$.
Let $u \colon \R \to (0, \infty)$ be a smooth function satisfying
\begin{align*}
  \begin{cases}
    u(\xi) \leq \sqrt{2 - \xi^2}, &\abs{\xi} \leq \sqrt{1 + \frac{\sqrt{2}}{2}}, \\
    u(\xi) = \frac{1}{\abs{\xi}}, &\abs{\xi} \geq \sqrt{1 + \frac{\sqrt{2}}{2}}.
  \end{cases}
\end{align*}
Then the graph of $u$ lies inside of $\psi^{-1}(W)$.
Let
\begin{align*}
  \wt{\mu}_1 \colon \mathbb{H}^2 &\to \R; \\
  (\xi, \eta) &\mapsto \frac{2\eta}{u(\xi)}.
\end{align*}
Then the range of $\wt{\mu}_1$ is $[0, \infty)$ and its level set over $0$ is $\Set{0} \times \R$.
Note that the Hamiltonian vector field $X_{\tilde{\mu}_1}$ is nonvanishing and its $\partial/\partial \xi$-component is $-\frac{\partial\tilde{\mu}_1}{\partial \eta} = 2\abs{\xi}$ for $\abs{\xi}$ large, so $X_{\tilde{\mu}_1}$ is complete.
Let $\mu_1 \colon \R^2 \to \R$ be uniquely defined by $\mu_1 \circ \psi = \wt{\mu}_1$, then $X_{\mu_1}$ is complete on $\R^2 \setminus L$, too.
This shows that $\mu_1$ satisfies \cref{ass:mu-1}, see \cref{fig:mu1}.

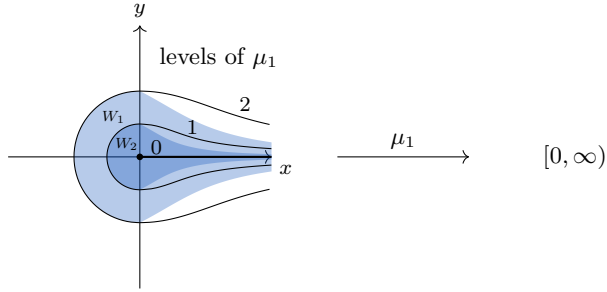
\inputfigure{mu1}{mu1}{Illustration of $\mu_1$ by its level sets.}

\bigskip

Let 
\begin{align*}
  W_1 &= \Set{(x, y) \in \R^2 \mmid x^2 + y^2 < 1\text{ when } x \leq 0\text{ or }\abs{y}\Pa{\Pa{x+1}^2 + y^2} < 2\text{ when } x \geq 0}, \\
  W_2 &= \Set{(x, y) \in \R^2 \mmid x^2 + y^2 < \frac14\text{ when } x \leq 0\text{ or }\abs{y}\Pa{\Pa{x+\frac12}^2 + y^2} < \frac14\text{ when } x \geq 0},
\end{align*}
so that $\overline{W_2} \subset W_1 \subset \overline{W_1} \subset W$, $\mu_1(\partial W_2) = (0, 1]$, and $\mu_1(\R^2 \setminus W_1) = (0, \infty)$.

We define $\mu_{\frac43} \colon \overline{W_2} \to \R$ by the following steps.
For any $(x, y) \in \overline{W_2}$, let $a = x + \sqrt{\frac14 - y^2}$.
Then denote by $(x_a, y_a)$ the unique intersection of $x_a + \sqrt{\frac14 - y_a^2} = a$ and $\abs{y_a}\Pa{\Pa{x_a+\frac12}^2 + y_a^2} = \frac14$.
The value of $\mu_{\frac43}(x, y)$ is defined as $\mu_1(x_a, y_a)$.

Let $\mu_{\frac53} \colon \R^2 \to \R$ be a smooth function without critical points which coincides with $\mu_1$ in $\R^2 \setminus W_1$ and with $\mu_{\frac43}$ in $\overline{W_2}$.
Note that the range of $\mu_{\frac53}$ is $(0, \infty)$, and the Hamiltonian vector field $X_{\mu_{\frac53}}$ is complete since it coincides with $X_{\mu_1}$ outside of $\overline{W_1}$ on which $\mu_{\frac53}$ is proper.
Compose with $\mu_{\frac53}$ a diffeomorphism $(0, \infty) \to \R$ which is the minus-one map in $[1, \infty)$ we obtain $\mu_2 \colon \R^2 \to \R$ satisfying \cref{ass:mu-2}, see \cref{fig:mu2}.

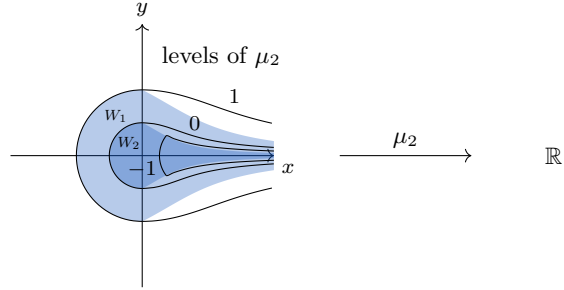
\inputfigure{mu2}{mu2}{Illustration of $\mu_2$ by its level sets.
The open sets $W_1$ and $W_2$ are used to construct $\mu_2$ from $\mu_1$; in the complement of $\overline{W_2}$, $\mu_2 = \mu_1 - 1$.}

The following proposition indicates that we will not lose generality by replacing $W(\varepsilon)$ by $\R^2$ in the definition of $\gamma$-width in \cref{sec:intro}.

\begin{proposition} \label{prop:embedding-R2}
  There is a symplectic embedding $(\R^2, \omega_0) \to (W, \omega_0)$ preserving $L_0$.
\end{proposition}

\begin{proof}
  The map
  \begin{align*}
    \R \times [0, 1) &\to \psi^{-1}(W) \subset \mathbb{H}^2; \\
    (\xi, \eta) &\mapsto \Phi_\mu^\xi(0, \eta)
  \end{align*}
  is a symplectic embedding preserving $L$, by explicit calculations or by \cite[Lemma~2.3]{2018arXiv180300998P}.
  By composing a symplectomorphism $\mathbb{H}^2 \to \R \times [0, 1)$ preserving $L$ with it we obtain $\tilde{\varphi}_0 \colon \mathbb{H}^2 \to \psi^{-1}(W)$ a symplectic embedding preserving $L$.
  Let $\varphi_0 \colon \R^2 \to W$ be the continuous injection preserving $L_0$ defined by $\psi \circ \tilde{\varphi}_0 = \varphi_0 \circ \psi$.
  Then replacing $\varphi_0$ in a neighborhood of $L_0$ by the identity we obtain a symplectic embedding of $\R^2$ into $W$ preserving $L_0$.
\end{proof}

\subsection{Diffeomorphisms} \label{ssec:diffeomorphism}

We show the existence of the maps satisfying \cref{ass:g1,ass:g3} by explicit constructions.

There is a smooth map $g_2 \colon B_1 \to \R$ satisfying the following assumption.

\begin{assumption} \label{ass:g2}
\ 
\begin{enumerate} [label={(\arabic*)}]
  \item $\frac{\partial g_2}{\partial c_n}(c_1, \dotsc, c_n) > 0$ if $\ce > 0$;
  \item $\displaystyle\lim_{c_n \to 0^+} g_2(0, \dotsc, 0, c_n) = 1$, and $g_2(c_1, \dotsc, c_{n-1}, 0) = 0$ for $\ce > 0$;
  \item $\displaystyle\lim_{c_n \to \infty} g_2(c_1, \dotsc, c_n) = \infty$.
\end{enumerate}
\end{assumption}

An example of $g_2$ satisfying \cref{ass:g2} is 
\begin{align*}
  g_2(c_1, \dotsc, c_n) &= \frac{1}{2}\Pa{c_n - \frac{1}{c_n}\frac{\ce}{\ce+1} + \sqrt{\Pa{c_n - \frac{1}{c_n}\frac{\ce}{\ce+1}}^2 + 4}} \\
  &= \frac{2c_n}{-\Pa{c_n^2 - \frac{\ce}{\ce+1}} + \sqrt{\Pa{c_n^2 - \frac{\ce}{\ce+1}}^2 + 4c_n^2}}.
\end{align*}
Note that the two formulae coincide in their common domain and they put together can be defined smoothly in an open neighborhood of $B_1$ in $\R^n$.

In order to find $g_1$ satisfying \cref{ass:g1} for a given $g_2$ satisfying \cref{ass:g2}, let $C_1 = \Set{(c_1, \dotsc, c_n) \in B_1 \mmid c_n \leq 1, \ce \leq \frac12 c_n}$.
Note that $C_1 \subset U_1$ and $\ce < g_2 < 2$ on $C_1$.
Define $g_{\frac32} \colon C_1 \cup (B_1 \setminus U_1) \to \R$ as $g_2$ on $C_1$ and $c_n$ on $B_1 \setminus U_1$.
Then $g_{\frac32}$ is increasing in $c_n$.
Let $g_1 \colon B \to \R$ be a smooth extension of $g_{\frac32}$ with $\frac{\partial g_1}{\partial c_n} > 0$, so $g_1$ satisfies \cref{ass:g1}.
One of the approaches to finding such a $g_1$ is to first extend $\frac{\partial g_{\frac32}}{\partial c_n}$ to a smooth positive function $B \to \R$, and then adjust the function so that its $c_n$-antiderivative starting at $0$ on $(c_1, \dotsc, c_{n-1}, 0)$ coincides with $g_{\frac32}$.

\bigskip

There is a smooth map $g_4 \colon B_3 \to \R$ satisfying the following assumption.

\begin{assumption} \label{ass:g4}
\ 
\begin{enumerate} [label={(\arabic*)}]
  \item $\frac{\partial g_4}{\partial c_n}(c_1, \dotsc, c_n) > 0$;
  \item $\displaystyle\lim_{c_n \to 0^+} g_4(0, \dotsc, 0, c_n) = \lim_{c_n \to -\infty} g_4(c_1, \dotsc, c_n) = -\infty$ if $\ce > 0$;
  \item $\displaystyle\lim_{c_n \to \infty} g_4(c_1, \dotsc, c_n) = \infty$.
\end{enumerate}
\end{assumption}

An example of $g_4$ satisfying \cref{ass:g4} is
\begin{align*}
  g_4(c_1, \dotsc, c_n) &= \Pa{1 + \frac{1}{\ce}} c_n - \sqrt{\frac{1}{\ce^2} c_n^2 + \frac{1}{\ce(2 + \ce)}} \\
  &= \frac{(2 + \ce) c_n^2 - \frac{1}{2 + \ce}}{(2 + \ce) c_n +\sqrt{c_n^2 + \frac{\ce}{2 + \ce}}}.
\end{align*}
Note that the two formulae coincide in their common domain and they put together can be defined smoothly in an open neighborhood of $B_3$ in $\R^n$.

Let $C_2 = \Set{(c_1, \dotsc, c_n) \in B_3 \mmid c_n \leq \frac12, \ce \leq \frac12}$.
Note that $C_2 \subset U_2$ and $g_4 < 1$ on $C_2$.
Define $g_{\frac72} \colon C_2 \cup (B_3 \setminus U_2) \to \R$ as $g_4$ on $C_2$ and $c_n$ on $B_3 \setminus U_2$.
Then $g_{\frac72}$ is increasing in $c_n$.
Let $g_3 \colon B \to \R$ be a smooth extension of $g_{\frac72}$ such that $\frac{\partial g_3}{\partial c_n} > 0$, so $g_3$ satisfies \cref{ass:g3}.

\begin{proof} [Proof of \cref{thm:main}]
  In \cref{ssec:momentum-map} we construct a $\mu_1$ satisfying \cref{ass:mu-1} and a $\mu_2$ satisfying \cref{ass:mu-2}.
  In \cref{ssec:diffeomorphism} we find a $g_1$ satisfying \cref{ass:g1} and a $g_3$ satisfying \cref{ass:g3}.
  The existence of such $\mu_1$ and $g_1$ validates \cref{lem:first-symplecto} and the existence of such $\mu_3$ and $g_3$ justifies \cref{lem:second-symplecto}

  Then \cref{thm:main} for $\varepsilon = 1$ is the result of \cref{lem:first-symplecto,lem:second-symplecto} and by rescaling we obtain the case of a general $\varepsilon$.
\end{proof}

\section{Cotangent bundles have nowhere vanishing primitives} \label{sec:cotangent-bundle}

In this section, we prove any cotangent bundle permits a nowhere vanishing primitive.
This gives a collection of examples of hamiltonian Lie algebroids which are defined and discussed in \cite{2018arXiv181111109B}.

As an application of \cref{thm:main}, we prove a property for any cotangent bundle below.
The proof of \cref{prop:cotangent-bundle} when $X$ is noncompact or compact with zero Euler characteristic was proved in \cite[Proposition~6.17]{2018arXiv181111109B}.
Note that $X$ is a Liouville vector field on $(M, \omega)$ if and only if $\iota_X \omega$ is a primitive of $\omega$.

\begin{proposition} [{\cref{cor:nonvanishing-Liouville-cotangent-bundle}}] \label{prop:cotangent-bundle}
  Let $N$ be a smooth manifold, there is $\alpha \in \Omega^1(T^*N)$ such that $\alpha$ is nowhere zero and $\der \alpha = \omega_{\mathrm{can}}$, the canonical symplectic form on $T^*N$.
\end{proposition}

\begin{proof}
  Let $M = T^*N$ and $n = \dim N$.
  Let $\theta$ be the tautological $1$-form on $M$.
  Let $\pi \colon M \to N$ be the projection.
  Then $\der \theta = -\omega_{\mathrm{can}}$ and $\theta(\xi) = \pi^* \xi$ for any $\xi \in T^*N$.

  Suppose $N$ is compact with nonzero Euler characteristic.
  Let $X$ be a smooth vector field on $N$ with finitely many zeroes, for instance, the Hamiltonian vector field of a Morse function.
  Let $f = \iota_X\theta \colon M \to \R$ and $\varepsilon > 0$.
  Let $\beta = \varepsilon \der f - \theta$.
  In local coordinates $(x_i, \xi_i)$, let $X = \sum_{i = 1}^n X_i \frac{\partial}{\partial x_i}$, then $\theta = \sum_{i = 1}^n \xi_i \der x_i$ and
  \begin{equation*}
    \beta = \sum_{i = 1}^n \Pa{\varepsilon \frac{\partial X_i}{\partial x_i} - 1} \xi_i \der x_i + \sum_{i = 1}^n \varepsilon X_i \der \xi_i.
   \end{equation*}
  Since $N$ is compact, we can choose $\varepsilon > 0$ such that $\varepsilon \frac{\partial X_i}{\partial x_i} < 1$ for any $i$ for any coordinate chart.
  Thus the set of zeroes of $\beta$ is $\operatorname{Fix}(X)$, the set of fixed points of $X$ on $N \simeq$ the zero section of $M = T^*N$.

  Now for any $c \in \operatorname{Fix}(X)$, take a linear ray $\gamma_c$ leaving $c$ in $T^*_cN$.
  Since $\operatorname{Fix}(X)$ is finite, there are disjoint neighborhoods $U_c$ of $c \in \operatorname{Fix}(X)$ in $N$.
  By applying \cref{thm:main} to $\gamma_c$ with neighborhood $\pi^{-1}(U_c)$ one by one, $M$ is symplectomorphic to $M' = M \setminus \bigcup_{c \in \operatorname{Fix}(X)} \image \gamma_c$ with respect to $\omega_{\mathrm{can}}$.
  Let $\alpha$ be the pullback of $\Res{\beta}_{M'}$ to $M$.
  Then $\der \alpha = \omega_{\mathrm{can}}$ since $\der \beta = \omega_{\mathrm{can}}$, and $\alpha$ has no zero left.
\end{proof}

This shows, by \cite[Proposition~6.15]{2018arXiv181111109B}, that the tangent Lie algebroid of any cotangent bundle is hamiltonian.

\bibliographystyle{hplain}
\bibliography{ref}

\authaddresses

\end{document}